\documentclass[12pt]{amsart}
\usepackage{amsmath,amssymb}
\usepackage{graphicx,bm}
\usepackage[margin=1in]{geometry}
\usepackage{hyperref}

\numberwithin{equation}{section}

\newtheorem{thm}{Theorem}[section]
\newtheorem{mthm}{Theorem}

\newtheorem{lem}[thm]{Lemma}
\newtheorem{prop}[thm]{Proposition}

\newtheorem{defn}[thm]{Definition}

\newcommand{\pt}{\partial}

\begin{document}

\title[Sharp Height estimate for CMC Spacelike Graph]{A Sharp Height Estimate for the Spacelike Constant Mean Curvature  Graph in the Lorentz-Minkowski Space }
\author{Jingyong Zhu}
\address{School of Mathematical Sciences,  University of Science and Technology of China}
\email{zjyjj0224@gmail.com}

\begin{abstract}
  In this paper, based on the local comparison principle in \cite{chen1982convexity}, we study the local behavior of the difference of two spacelike graphs in a neighborhood of a second contact point. Then we apply it to the constant mean curvature equation in 3-dimensional Lorentz-Minkowski space $\mathbb{L}^3$ and get the uniqueness of  critical point for the solution of such equation over convex domain, which is an analogue of the result in \cite{sakaguchi1989uniqueness}. Last, by this uniqueness, we obtain a minimum principle for a functional depending on the solution and its gradient. This gives us a sharp gradient estimate for the solution, which leads to a sharp height estimate.
\end{abstract}
\maketitle
\section{Introdution}    
Spacelike constant mean curvature(CMC) hypersurfaces and CMC foliation play an important role in general relativity.  Such surfaces are important because they provide Riemannian submanifolds
with properties which reflect those of the spacetime. For example, if the weak energy condition is satisfied, then a maximal hypersurface has positive scalar curvature. So the geometric properties of such hypersurfaces are worth researching. In particular, the existence of such hypersurface is a fundamental problem. Under the graph setting and some assumptions, Robert Bartnik and Leon Simon\cite{bartnik1982spacelike} got a sufficient and necessary condition for the existence of  
\begin{equation}
\begin{cases}
  \mathrm{div}(\frac{Du}{\sqrt{1-|Du|^2}})=H(x,u),\ \  |Du|<1 \ \text{in} \ \Omega\subset\mathbb{R}^n,  \\
  u=\phi \ \ \ \text{on} \ \ \pt\Omega,
\end{cases}
\end{equation}
where $\mathrm{div}$ stands for divergence operator in the Euclidean plane $\mathbb{R}^n$ and 
\begin{equation}
Du=(u_1,\dots,u_n), \  \  u_i=\frac{\partial u}{\partial x_i}.
\end{equation} In particular, the Theorem 3.6 in \cite{bartnik1982spacelike} gives us a solution $u\in C^{\infty}(\overline\Omega)$ to
\begin{equation}
\begin{cases}
  \mathrm{div}(\frac{Du}{\sqrt{1-|Du|^2}})=nH,\ \  |Du|<1 \ \text{in} \ \Omega,  \\
  u=0 \ \ \ \text{on} \ \ \pt\Omega,
\end{cases}
\end{equation}
over a bounded $C^{2,\alpha}$ domain $\Omega$ with $H$ being a positive constant. In this case, they pointed out that $\nu_{n+1}=\frac{1}{\sqrt{1-|Du|^2}}$ satisfies following elliptic equation
\begin{equation}
\Delta_M\nu_{n+1}=\nu_{n+1}\|A\|^2+\frac{1}{1-|Du|^2}\frac{\partial u}{\partial x_i}\frac{\partial H}{\partial x_i},
\end{equation}
where $\Delta_M$ and $A$ denote the Laplace operator and the second fundamental form of the graph $M=\{(x,u(x)):x\in\mathbb{R}^n, u\in C^{\infty}(\mathbb{R}^n)\}$. So the boundary gradient estimate is the most important step leading to the existence of $u$. To do so, they used the following spherically symmetric barrier functions 
\begin{equation}
w^{\pm}=w^{\pm}(\xi)\pm\int_0^{|x-\xi|}\frac{K-Ht^n}{\sqrt{t^{2n-2}+(K-Ht^n)^2}}\rm{d}t
\end{equation}
where $K$ is a positive constant. From the proof of the Proposition 3.1, one can get following boundary gradient estimate
\begin{equation}
\max\limits_{\partial \Omega}|Du| \leq\frac{1-H\varepsilon^{n+1}}{\sqrt{\varepsilon^{2n}+(1-H\varepsilon^{n+1})^2}},
\end{equation}
where $\varepsilon=\varepsilon(\Omega)$ is a sufficiently small constant. Obviously, this bound is not sharp. Also, the dependence of $\varepsilon$ on $\Omega$ is not specific. As for the $C^0$ norm of the solution $u$, since the graph is spacelike, they roughly used the diameter of the domain $\Omega$ to control it. So the question is, can we give a sharp $C^0$ or $C^1$ estimate for the solution in term of the boundary geometry?

Early in 1979,  Lawrence E. Payne and G{\'e}rard A. Philippin\cite{payne1979some} have used so-called P-function to derive sharp $C^0$ and $C^1$ upper bounds for the solution of following Dirichlet problem
\begin{equation}\label{eq1.6}
\begin{cases}
  \mathrm{div}(\frac{Du}{\sqrt{1+|Du|^2}})=-2H,\ \  |Du|<1 \ \text{in} \ \Omega,  \\
  u=0 \ \ \ \text{on} \ \ \pt\Omega,
\end{cases}
\end{equation}
over a strictly convex domain $\Omega\subset\mathbb{R}^2$ with $H$ being a positive constant. The key is a maximum principle for following P-function
\begin{equation}
\Phi(x, \alpha)=\int_0^{q^2}\frac{g(\xi+2\xi g'(\xi))}{\rho}\rm{d}\xi+\alpha\int_0^uf(\eta)\rm{d}\eta,
\end{equation}
where $u, g, \rho, f, q$ satisfy
\begin{equation}\label{}
\begin{split}
&(g(q^2)u_i)_i+\rho(q^2)f(u)=0,\\
&g(\xi)+2xig'(xi)>0,  \  \text{for}  \  \forall\xi\geq0,\\
&\rho>0, \  \   g>0, \  \   q^2=|Du|^2=\sum u_i^2.
\end{split}
\end{equation} 
In the same year, by the uniqueness of critical point for solution and the strictly convexity of the domain, G. A. Philippin\cite{philippin1979minimum} also got a minimum principle for $\Phi(x, \alpha)$ provoided $\alpha>1$ and used it to derive lower bounds for $C^0$ and $C^1$ norms of the solution. But he did not assert the sharpness of the estimates, since he did not have a similar minimum principle for $\Phi(x, 1)$  at that time. In 2000, Xi-Nan Ma\cite{ma2000sharp} solved this issue through uniqueness of critical point and analyticity of the solution. He took a long computation to show that all the derivatives of $\Phi(x, 1)$ vanish at the unique critical point if $\Phi(x, 1)$ takes its minimum value at that point. By the strong unique continuation of analytic function, $\Phi(x, 1)$ is a constant. Once has this minimum principle, one can get the sharpness. 

For our question, the maximum principle in \cite{payne1979some} still works. So the upper bound of gradient estimate and the lower bound of the minimum value are easy to derive, which will be given in this paper later. However, the minimum principle is not available any more. In this paper, the author want to prove a minimum principle for $\Phi(x, 1)$ when $u$ is a spacelike CMC graph solving 
\begin{equation}\label{eq13}
\begin{cases}
  \mathrm{div}(\frac{Du}{\sqrt{1-|Du|^2}})=2H,\ \  |Du|<1 \ \text{in} \ \Omega,  \\
  u=0, \ \ \ \text{on} \ \ \pt\Omega,
\end{cases}
\end{equation}
and use it to derive sharp $C^0$ and $C^1$ bounds for the solution to $\eqref{eq13}$.

As we mentioned above, the uniqueness of critical point is the important ingredient which itself is also worth to be studied. There are some results on it. In \cite{philippin1979minimum}, G. A. Philippin showed that the solution to $\eqref{eq1.6}$ has only one critical point when $\Omega$ is strictly convex. His method of proof is based on an idea of L. E. Payne\cite{payne1973two}.  In \cite{chen1984uniqueness}, Jin-Tzu Chen proved that the uniqueness of critical point for solution to
 \begin{equation}\label{eq15}
\begin{cases}
  \mathrm{div} \ Tu=2H\ \ \ \text{in} \ \Omega,  \\
  Tu\cdot\nu=1,   \  \text{on} \ \ \pt\Omega,
\end{cases}
\end{equation}
where $\Omega\subset\mathbb{R}^2$ is a bounded convex domain with out normal $\nu$ on the boundary $\pt\Omega$, $H$ is a positive constant,  and
\begin{equation}\label{eq12}
Tu=(\frac{u_x}{\sqrt{1+u_x^2+u_y^2}},\frac{u_y}{\sqrt{1+u_x^2+u_y^2}}),\ \ \ Du=(u_x,u_y),
\end{equation}
$u_x$,$u_y$ being partial derivatives. His proof is based on a nice comparison technique and the result   in \cite{chen1982convexity} and the method of continuity with respect to the contact angle. Later, Shigeru Sakaguchi\cite{sakaguchi1989uniqueness} showed that the solution to
 \begin{equation}\label{eq14}
\begin{cases}
  \mathrm{div} \ Tu=2H\ \ \ \text{in} \ \Omega,  \\
  u=0, \ \ \ \text{on} \ \ \pt\Omega,
\end{cases}
\end{equation}
or
 \begin{equation}\label{eq15}
\begin{cases}
  \mathrm{div} \ Tu=2H\ \ \ \text{in} \ \Omega,  \\
  Tu\cdot\nu=\cos\gamma, \ \ \gamma\in(0,\frac{\pi}{2}) \ \text{on} \ \ \pt\Omega,
\end{cases}
\end{equation}
has only one critical point under the hypothesis of the existence of the solution over a bounded convex domain $\Omega\subset\mathbb{R}^2$.

 Another motivation to study uniqueness of critical point for solution to $\eqref{eq13}$  is from a recent paper \cite{albujer2015convexity}.  As we know, CMC spacelike hypersurfaces are very different to those in Euclidean space. For example, Corollary 12.1.8 in \cite{lopez2013constant} tells us any compact spacelike surface immersed in $\mathbb{L}^3$ spanning a plane simple closed curve is a graph over a spacelike plane, which is not true in $\mathbb{R}^3$. Therefore, up to an isometry, we only need to consider the solution to the Dirichlet problem $\eqref{eq13}$ . Recently, Alma L. Albujer, Magdalena Caballero and Rafael L{\'o}pez\cite{albujer2015convexity} have proved the following interesting theorem on the convexity of the solutions to $\eqref{eq13}$.
  \begin{mthm}\label{thm A}\cite{albujer2015convexity}
  Let $\Sigma$ be a spacelike compact surface in $\mathbb{L}^3$ with constant mean curvature $H\neq0$(H-surface for short), such that its boundary is a planar curve which is pseudo-elliptic. Then $\Sigma$ has negative Gaussian curvature in all its interior points. In particular, $\Sigma$ is a convex surface.
  \end{mthm}
In their paper, they also proved that pseudo-elliptic curves are convex and provided an example that shows the assuption on the boundary can not be replaced by convex curves, but they did not show whether there is a critical point of the solution to $\eqref{eq13}$ with nonnegative Gaussian curvature over a convex domain, which is so-called saddle point? In this paper, we will show that the non-existence of such saddle point is equivalent to the uniqueness of critical point. Notice that the Gaussian curvature in \cite{sakaguchi1989uniqueness} is different from that in the Theorem \ref{thm A}, which is defined in the next section. 

Now, let us state our first result.
       
\begin{thm}\label{thm11}
 Any solution to $\eqref{eq13}$ in a convex domain for $H\neq0$ has only one critical point.
 \end{thm}  

 The proof of this theorem is based on the idea of Shigeru Sakaguchi in \cite{sakaguchi1989uniqueness}, where mainly relies on the comparison of a cylinder with the given surface and the continuity method. In the present result, our comparison surface is a connected component of a hyperbolic cylinder, which is an entire graph over $\mathbb{R}^2$ and, in contrast with the Euclidean case, the existence of the solution for any bounded domain is assured by the necessary and sufficient conditions given by Robert Bartnik and Leon Simon\cite{bartnik1982spacelike}.

 As we said before, Theorem \ref{thm11} can be used to derive sharp $C^0$ and $C^1$ bounds for the solution to $\eqref{eq13}$.
 \begin{thm}\label{thm12}
 Let $u\in C^{\infty}(\overline{\Omega})$ be a solution to $\eqref{eq13}$ over a strictly convex domain $\Omega$ for $H\neq 0$ and $K$ be the curvature of the boundary $\pt\Omega$ with respect to the inner normal direction. Then
 \begin{equation}\label{}
 \max\limits_{\overline\Omega}|Du|^2=\max\limits_{\partial\Omega}|Du|^2\leq\frac{H^2}{H^2+K_{min}^2},
 \end{equation}
 \begin{equation}\label{eq16}
 -\frac{1}{H}(\frac{\sqrt{H^2+K_{min}^2}}{K_{min}}-1)\leq\min_{\Omega}u\leq-\frac{1}{H}(\frac{\sqrt{H^2+K_{max}^2}}{K_{max}}-1)
 \end{equation}
 where $K_{min}=\min\limits_{\pt \Omega}K$, $K_{max}=\max\limits_{\pt \Omega}K$, and one of the equality signs holds if and only if the boundary $\pt\Omega$ is a circle.
 \end{thm}  

At this point, we should give a remark. When $H\neq 0$ and $\Omega$ is a round disc of radius R (which is centered at the origin), then
\begin{equation}
u(x,y)=\sqrt{x^2+y^2+\frac{1}{H^2}}-\sqrt{R^2+\frac{1}{H^2}},
\end{equation}
whose graph is a so-called hyperbolic cap\cite{lopez2013constant}.

 This article is organized as follows. In section 3, we will investigate the local behavior of the difference of two spacelike graphs in a neighborhood of a second contact point. In section 4, we will prove a necessary and sufficient condition for uniqueness of minimal point of the solution to $\eqref{eq13}$, which is a key step in the proof of the Theorem \ref{thm11} in section 5. In the end, based on the uniqueness of critical point, we will prove a minimum principle and use it to get the sharp estimates in Theorem \ref{thm12}.

\section{Notions and local comparison technique}
For easier reading, let us recall some background knowledge of Lorentzian geometry. More details can be found in \cite{lopez2013constant}. Let $\mathbb{L}^3$ be the 3-dimensional Lorentz-Minkowski space, that is $\mathbb{R}^3$ endowed with the flat Lorentzian metric $$\langle\cdot,\cdot\rangle=dx_{1}^{2}+dx^2_2-dx_3^2,$$where $(x_1,x_2,x_3)$ are the canonical coordinates in $\mathbb{R}^n$. The non-degenerate metric of index one classifies the vectors of $\mathbb{R}^3$ into three types.
\begin{defn}\cite{lopez2013constant}
A vector $v\in\mathbb{L}^3$ is said to be:\\
1. spacelike if $\langle v,v\rangle>0$ or $v=0$;\\
2. timelike if $\langle v,v\rangle<0$;\\
3. lightlike if $\langle v,v\rangle=0$ and $v\neq0$.
\end{defn}
The modulus of $v$ is $|v|=\sqrt{|\langle v,v\rangle|}$.
\begin{defn}\cite{lopez2013constant}
An immersed surface $\Sigma$ in $\mathbb{L}^3$ is called spacelike if the induced metric on $\Sigma$ is positive-definite.
\end{defn}

Given a spacelike immersed surface $\Sigma$, by Proposition 12.1.5 in \cite{lopez2013constant}, $\Sigma$ is orientable. We can choose $\Sigma$ to be future oriented that means the unit normal vector field $N$ satisfying $\langle N, e_3\rangle>0$. Here $e_3=(0,0,1)$. Let $\overline{\nabla}$ and $\nabla$ denote the Levi-Civita connection in $\mathbb{L}^3$ and $\Sigma$, respectively. If $X,Y\in\mathfrak{X}(\Sigma)$, the Gauss and Weingarten formulae are
\begin{equation}
\overline{\nabla}_XY=\nabla_XY+\sigma(X,Y)=\nabla_XY-\langle AX,Y\rangle N
\end{equation}
and
\begin{equation}
AX=-\overline{\nabla}_XN,
\end{equation}
respectively, where $\sigma$ is the second fundamental form and $A:\mathfrak{X}(\Sigma)\rightarrow\mathfrak{X}(\Sigma)$ stands for the shape operator of $\Sigma$ with respect to $N$. The mean curvature and the Gaussian curvature are defined by
\begin{equation}
H=-\frac{1}{2}\mathrm{trace}(A)=-\frac12(\kappa_1+\kappa_2),\ \ \ K=-\rm{det}(A)=-\kappa_1\kappa_2.
\end{equation}

Let $u\in C^2(\Omega)$ be a function defined on a domain $\Omega\in\mathbb{R}^2$ and consider the surface $\Sigma_u=(x,y,u(x,y))$. The coefficients of the first fundamental form are
\begin{equation}
E=1-u_x^2, \ \ \ F=-u_xu_y, \ \ \ G=1-u_y^2.
\end{equation}
Thus $EG-F^2=1-u_x^2-u_y^2=1-|\nabla u|^2$ and since the immersion is spacelike, $|\nabla u|^2<1$ on $\Omega$. The future-directed normal is given by
\begin{equation}
N(x,y,u(x,y))=\frac{(u_x,u_y,1)}{\sqrt{1-|\nabla u|^2}}=\frac{(\nabla u,1)}{\sqrt{1-|\nabla u|^2}}.
\end{equation}
With this normal, the mean curvature $H$ and Gaussian curvature $K$ satisfy
\begin{equation}\label{eq13'}
\mathrm{div}\left(\frac{(\nabla u,1)}{\sqrt{1-|\nabla u|^2}}\right)=2H
\end{equation}
and
\begin{equation}\label{eq K}
K=-\frac{u_{xx}u_{yy}-u_{xy}^2}{(1-|\nabla u|^2)^2},
\end{equation}
respectively, where $\rm{div}$ is the Euclidean divergence in $\mathbb{R}^2$.

As mentioned previously, every compact spacelike surface $\Sigma$ in $\mathbb{L}^3$ with simple closed boundary contained in a hyperplane can be regarded as the graph of a solution $u(x,y)$ to $\eqref{eq13}$. There are more interesting facts on compact spacelike surfaces in $\mathbb{L}^3$ with constant mean
curvature spanning a given boundary curve(see \cite{lopez2013constant}).

From now on, we assume $u$ to be a solution to $\eqref{eq13}$ with $H>0$ in a convex domain $\Omega$. For $H<0$, we can consider $-u$ and our theorem still holds. By the maximum principle, $u$ has a interior minimal point, which is a point of nonpositive Gaussian curvature.

In the rest of this section, based on the local comparison technique found in \cite{chen1982convexity}, we will investigate the local behavior of the difference of two spacelike graphs in a neighborhood of the  point, where they have the second contact.
\begin{lem}\label{lem21}
Let $u(x,y), v(x,y)$ satisfy the same spacelike constant mean curvature equation(the first equation in $\eqref{eq13}$ or $\eqref{eq13'}$). Without loss of generality, we assume that $u,v$ have a second-order contact at $P_0=(x_0,y_0,u(x_0,y_0))$ with $(x_0,y_0)=(0,0)$. Then by changing coordinate $(x,y)$ into $(\xi,\eta)$ linearly, the difference $u-v$ around $(\xi,\eta)=(0,0)=(x,y)$ is given by
 \begin{equation}\label{eq21}
 u-v=\mathbb{Re}(\lambda\cdot(\xi+\eta i)^n+o(\xi^2+\eta^2)^{\frac{n}{2}}),
\end{equation}
where $n\geq3$, $\lambda$ is a complex number and $\xi+\eta i$ is the complex coordinate.
\end{lem}
\begin{proof}
Let $w=u-v$. Since $u,v$ solve the same constant mean curvature equation, we have
\begin{equation}\label{eq22}
\begin{split}
0&=(1-u_x^2-u_y^2)(u_{xx}+u_{yy})+(u_x^2u_{xx}+u_y^2u_{yy}+2u_xu_yu_{xy})-2H(\sqrt{1-|Du|^2})^3\\
&=(1-u_y^2)u_{xx}+(1-u_x^2)u_{yy}+2u_xu_yu_{xy}-2H(\sqrt{1-u_x^2-u_y^2})^3,
\end{split}
\end{equation}
\begin{equation}\label{eq23}
0=(1-v_y^2)v_{xx}+(1-v_x^2)v_{yy}+2v_xv_yv_{xy}-2H(\sqrt{1-v_x^2-v_y^2})^3.
\end{equation}
Define $r(\tau),s(\tau),t(\tau),p(\tau),q(\tau)$ for $0\leq\tau\leq1$ by
\begin{equation}\label{eq24}
\begin{split}
&r(\tau)=(1-\tau)v_{xx}+\tau u_{xx},\ \ \ s(\tau)=(1-\tau)v_{xy}+\tau u_{xy},\ \ \ t(\tau)=(1-\tau)v_{yy}+\tau u_{yy},\\
&p(\tau)=(1-\tau)v_{x}+\tau u_{x},\ \ \ \ q(\tau)=(1-\tau)v_{y}+\tau u_{y},
\end{split}
\end{equation}
and consider the function
\begin{equation}\label{eq25}
F=F(\tau)=(1-q^2)r+2pqs+(1-p^2)t-2H(\sqrt{1-p^2-q^2})^3.
\end{equation}
Then we get
\begin{equation}\label{eq26}
\begin{split}
0&=F(1)-F(0)=\int_0^1\frac{\pt F}{\pt\tau}d\tau\\
&=a_{11}w_{xx}+2a_{12}w_{xy}+a_{22}w_{yy}+b_1w_x+b_2w_y,
\end{split}
\end{equation}
with
\begin{equation}\label{eq27}
\begin{split}
&a_{11}=\int_0^1(1-q^2)d\tau,\ \ a_{12}=\int_0^1pqd\tau,\ \ a_{22}=\int_0^1(1-p^2)d\tau,\\
&b_1=-2\int_0^1[(pt-qs)-3H\sqrt{1-p^2-q^2}p]d\tau,\\
&b_2=-2\int_0^1[(qr-ps)-3H\sqrt{1-p^2-q^2}q]d\tau.
\end{split}
\end{equation}
Since $Dw=0$ at $(0,0)$, there exists a neighborhood, say $O(0,0)$, such that $(p,q)$ stays in the unit ball, i. e. $p^2+q^2<1$ over $O(0,0)$. Therefore, we have
\begin{equation}\label{eq28}
\begin{split}
a_{12}^2&=(\int_0^1pqd\tau)^2
\leq\int_0^1(p^2)d\tau\int_0^1(q^2)d\tau\\
&<\int_0^1(p^2)d\tau\int_0^1(1-p^2)d\tau
<\int_0^1(1-q^2)d\tau\int_0^1(1-p^2)d\tau=a_{11}a_{22}.
\end{split}
\end{equation}
Hence, $w$ satisfies a homogeneous elliptic equation
\begin{equation}\label{eq29}
Lw=a_{11}w_{xx}+2a_{12}w_{xy}+a_{22}w_{yy}+b_1w_x+b_2w_y,
\end{equation}
in $O(0,0)$.

Now, we transform $(x,y)$ into $(\xi,\eta)$ such that $\xi(0,0)=0$ and $\eta(0,0)=0$ and at $(0,0)$
\begin{equation}\label{eq210}
Lw=\left(\frac{\pt^2}{\pt\xi^2}+\frac{\pt^2}{\pt\eta^2}+b'_1\frac{\pt}{\pt\xi}+b'_2\frac{\pt}{\pt\eta}\right)w.
\end{equation}
Since the coefficient of $Lw$ and $w$ itself are analytic in $(x,y)$ as well as in $(\xi,\eta)$, we have the expansion around $(\xi,\eta)=(0,0)$ as follows,
\begin{equation}\label{eq211}
\begin{split}
Lw=&\left\{(1+\alpha_{11}\xi+\beta_{11}\eta+O(\xi^2+\eta^2))\frac{\pt^2}{\pt\xi^2}
+2(\alpha_{12}\xi+\beta_{12}\eta+O(\xi^2+\eta^2))\frac{\pt^2}{\pt\xi\pt\eta}\right.\\
&\left.+(1+\alpha_{22}\xi+\beta_{22}\eta+O(\xi^2+\eta^2))\frac{\pt^2}{\pt\eta^2}
+(\gamma_1+\delta_1\xi+\lambda_1\eta+O(\xi^2+\eta^2))\frac{\pt}{\pt\xi}\right.\\
&\left.+(\gamma_2+\delta_2\xi+\lambda_2\eta+O(\xi^2+\eta^2))\frac{\pt}{\pt\eta}\right\}w.
\end{split}
\end{equation}
By the Theorem I in \cite{bers1955local}, we know
\begin{equation}
w=w(\xi,\eta)=P_n(\xi,\eta)+o(\xi^2+\eta^2)^{\frac{n}{2}},
\end{equation}
where $P_n(\xi,\eta)$ is a non-zero harmonic homogeneous polynomial in $(\xi,\eta)$ of degree $n$. We know $n\geq3$, as $u$ and $v$ have a second contact at $(0,0)$. Thus the argument in page 82 of \cite{axler2013harmonic} tells us
\begin{equation}\label{eq213}
P_n(\xi,\eta)=\mathrm{Re}(\lambda\cdot(\xi+\eta i)^n).
\end{equation}
where $\lambda$ is a complex number. This, together with $\eqref{eq211}$, completes the proof of the lemma.
\end{proof}

Let $u-v$ to be defined on $D\in\mathbb{R}^2$ and $Z$ be the zero set of $u-v$ extended to the closure $\overline{D}$ of $D$. By Lemma \ref{lem21}, $Z$ divides a neighborhood $U$ of $(0,0)$ into at least six components on which the sign of $u-v$ alternate. However, Lemms \ref{lem21} does not tell us that $Z\cap U$ is a union of smooth arcs intersecting at $(0,0)$. We do not know if $Z$ may contain cusps at $(0,0)$. To exclude such irregular possibilities, we need the Lemma 2 in \cite{chen1982convexity}.
\begin{lem}\cite{chen1982convexity}\label{lem22}
Let $f=f(x,y)$ be a non-constant solution of a homogeneous quasilinear elliptic equation of the form
\begin{equation}\label{eq214}
Lf=a_{11}f_{xx}+2a_{12}f_{xy}+a_{22}f_{yy}+b_1f_x+b_2f_y=0, \ \ \text{in} \ \Omega,
\end{equation}
having analytic coefficients $a_{ij}$'s and $b_k$'s in $x,y$ and involving no zero order term. Then every interior critical point of $f$ is an isolated critical point.
\end{lem}

Using the previous two lemmas as well as the implicit function theorem, we see that the zero set $Z\cap U$ of $u-v$ consists of at least three smooth arcs intersecting at $(0,0)$ and dividing $U$ into at least six sectors. Furthermore, the zero set $Z$ is globally a union of smooth arcs.

\section{Nonuniqueness of the minimal point}
In this section, by using Lemma \ref{lem21} and Lemma \ref{lem22}, we will prove a sufficient and necessary condition for the nonuniqueness of minimal point of the solutions $v_t(t\in[0,1])$ to
\begin{equation}\label{eq31}
\begin{cases}
  \mathrm{div}(\frac{Dv}{\sqrt{1-t^2|Dv|^2}})=2H,\ \  t|Dv|<1 \ \text{in} \ \Omega,  \\
  v=0, \ \ \ \text{on} \ \ \pt\Omega.
\end{cases}
\end{equation}
Let $u_t=tv_t$ for $t>0$. Then $u_t$ satisfy
\begin{equation}\label{eq32}
\begin{cases}
  \mathrm{div}(\frac{Du}{\sqrt{1-|Du|^2}})=2tH,\ \  |Du|<1 \ \text{in} \ \Omega,  \\
  u=0, \ \ \ \text{on} \ \ \pt\Omega.
\end{cases}
\end{equation}
\begin{prop}\label{prop31}
There always exists a unique solution $v_t$ to $\eqref{eq31}$ satisfying
\begin{equation}\label{eq33}
t|Dv_t|\leq 1-\theta_0<1, \ \text{in} \ \overline{\Omega}, \ \parallel v_t\parallel_{C^{2,\alpha}(\overline{\Omega})}\leq C, \ \text{for all} \ t\in[0,1],
\end{equation}
where $C,\theta_0,\alpha$ are positive constants independent of $t$.
\end{prop}
\begin{proof}
By Theorem 3.6 in \cite{bartnik1982spacelike}, Theorem 13.8 in \cite{trudinger1983elliptic} and Theorem 12.2.2 in \cite{lopez2013constant}, there is a unique solution $u_t\in C^{2,\alpha}(\overline{\Omega})$ to the problem $\eqref{eq32}$ with
\begin{equation}\label{eq34}
|Du_t|<1-\theta_0<1, \ \ \text{in} \ \ \overline{\Omega},\ \ \ \parallel u_t\parallel_{C^{2,\alpha}(\overline{\Omega})}\leq C
\end{equation}
where $C,\theta_0,\alpha$ are positive constants independent of $t$.

Put $v_t=t^{-1}u_t$. Then $v_t$ satisfies $\eqref{eq31}$. By putting
\begin{equation}\label{eq35}
b(x)=(1-|Du_t|^2)^{-\frac12},
\end{equation}
we regard $v_t$ as a unique solution to the linear elliptic Dirichlet problem:
\begin{equation}\label{eq36}
\begin{cases}
  \mathrm{div}(b(x)Dv_t)=2H,\ \  \ \text{in} \ \Omega,  \\
  v_t=0, \ \ \ \text{on} \ \ \pt\Omega.
\end{cases}
\end{equation}
In view of $\eqref{eq34}$, using the Schauder global estimate(see Theorem 6.6 in \cite{trudinger1983elliptic}), we get
\begin{equation}\label{eq37}
\parallel v_t\parallel_{C^{2,\alpha}(\overline{\Omega})}\leq C(\sup\limits_\Omega|v_t|+2H).
\end{equation}
Also, it follows from Theorem 3.7 in \cite{trudinger1983elliptic} that
\begin{equation}\label{eq38}
 \sup\limits_\Omega|v_t|\leq C.
 \end{equation}
 Therefore, we get $\eqref{eq33}$ for $t\in(0,1]$. In the case that $t=0$, $\eqref{eq31}$ is a linear problem. Hence there exists a unique solution $v_0\in {C^{\infty}(\overline{\Omega})}$ to $\eqref{eq31}$. This completes the proof.
\end{proof}

 Before proving the sufficient and necessary condition for nonuniqueness of the minimal point of $v_t$, we need the following lemmas.
 \begin{lem}\label{lem31}
 Let $t$ belong to (0,1]. If $Dv_t=0$ at some point $p\in\Omega$, then the Gaussian curvature $K_t(p)$ of the graph $\Sigma_{v_t}=(x,y,v_t(x,y)$ at $p$ does not vanish.
 \end{lem}
\begin{proof}
Since $t$ is positive, it suffices to show this for $u_t=tv_t$. Recall that graph of $u_t$ has constant mean curvature $tH$. Let $p$ be a critical point of $u_t$ with $K_t(p)=0$.

Consider the upper connected component of a hyperbolic cylinder in $\mathbb{L}^3$, $S$, with radius $r=\frac{1}{(2tH)}$, tengent to $\Sigma_{u_t}$ at $p$ and such that the line generators are parellel to the zero principal curvature direction of $\Sigma_{u_t}$ at $p$. Recall that each connected component of a hyperbolic cylinder is an entire graph over $\mathbb{R}^2$ with constant mean curvature $|H|=\frac{1}{(2r)}$ and zero Gaussian curvature.

In general, the intersection of $S$ and $\mathbb{R}^2$ should be a branch of a hyperbola or two parallel lines. In our case, it should be the latter one, as $S$ touches $u_t$ at its critical point $p$. Hence, $S\cap\mathbb{R}^2$ divides $\mathbb{R}^2$ into three domains and suppose that the piece of $S$ with negative height is the graph of a function $v\in C^\infty(\Omega'), v<0$.

Define $D=\Omega\cap\Omega'$. On the one hand, by the convexity of $\Omega$, we see $\pt(\Omega\cap\Omega')$ consists of at most four arcs, each of which belongs to $\pt\Omega$ or $\pt\Omega'$ alternatively. Consider $A=\{(x,y)\in\Omega\cap\Omega'|u_t(x,y)>v(x,y)\}$. Since $u_t=0$ on $\pt\Omega$ and $v=0$ on $\pt\Omega'$, there are at most two components of $A$, each of which meets the boundary $\Omega\cap\Omega'$. On the other hand, by previous construction, $u_t$ and $v$ have a second-order contact at $p$. Lemma \ref{lem21} and Lemma \ref{lem22} tell us $A$ has at least three components each of which meets $\Omega\cap\Omega'$. Thus we get a contradiction. This completes the proof.
\end{proof}

Now, we see that there is no critical point of $v_t$ with Gaussian curvature vanishing for $t\in(0,1]$. What about the case of $t=0$?
\begin{lem}\label{lem32}
Every critical point $p$ of $v_0$ is a minimal point, i.e. the Gaussian curvature $K_0(p)$ of the graph $\Sigma_{v_0}$ is negative at $p$.
\end{lem}
\begin{proof}
Let $p$ be a critical point of $v_0$. Then $K_0(p)=-((v_0)_{xx}(v_0)_{yy}-(v_0)_{xy}^2)$ by $\eqref{eq K}$. Suppose that $K_0(p)\geq0$. For simplicity, by translation and rotation of the coordinate, we may assume that $p=(0,0)$ and $[D_{ij}v_0]=\rm{diag}[\lambda_1,\lambda_2]$, where $\lambda_1+\lambda_2=2H>0$, $\lambda_1>0$ and $\lambda_2\leq0$. Then $v_0(x,y)=w(x,y)+P(x,y)$, where $w(x,y)=v_0(0,0)+\frac12\lambda_1x^2+\frac12\lambda_2y^2$ and $P(x,y)$ is a harmonic function in $\Omega$. Consider
\begin{equation}\label{eq39}
A=\{(x,y)\in\Omega|P(x,y)>0\}, \ \ \ B=\{(x,y)\in\Omega|P(x,y)<0\}.
\end{equation}
Since $P(x,y)$ vanishes up to second order derivatives at $(0,0)$ and $P(x,y)$ is real analytic, it follows from Lemma \ref{lem21} and Lemma \ref{lem22} that
\begin{equation}\label{eq310}
\begin{split}
&\text{Both} \ A \ \text{and} \ B \ \text{have at least three components}\\
&\text{each of which meets the boundary}  \  \pt\Omega.
\end{split}
\end{equation}
Put $\Omega'=\{(x,y)\in\mathbb{R}^2|w(x,y)<0\}$. Since $\Omega$ is convex and $w$ is a quadratic function with $\lambda_1>0$ and $\lambda_2\leq0$, we see that $\pt(\Omega\cap\Omega')$ consists of at most four arcs each of which belongs to $\pt\Omega$ or $\pt\Omega'$ alternatively. Let $A'=\{(x,y)\in\Omega\cap\Omega'|P(x,y)>0\}$. Since $v_0=0$ on $\pt\Omega$ and $w=0$ on $\pt\Omega'$, there are at most two components of $A'$ each of which meets the boundary $\pt(\Omega\cap\Omega')$. This contradicts $\eqref{eq310}$. This completes the proof.
\end{proof}

Now, we can prove the sufficient and necessary condition for nonuniqueness of the minimal point of $v_t$.
\begin{thm}\label{thm31}
Let $t$ belong to $[0,1]$. The solution $v_t$ has more than two minimal points if and only if there exists a saddle point $p\in\Omega$, i.e. $Dv_t(p)=0$ and $K_t(p)>0$.
\end{thm}
\begin{proof}
It follows from Hopf's boundary point lemma that $Dv_t\cdot\nu$ is positive on $\pt\Omega$. There $v_t$ does not have minimal point on the boundary $\pt\Omega$.

``$\text{\bf{if part}}$". Let $p\in\Omega$ be a point with $Dv_t(p)=0$ and $K_t(p)>0$. Then there exists an open neighborhood $U$ of $p$ in which the zero set of $\widetilde{v}_t=v_t-v_t(p)$ consists of two smooth arcs intersecting at $p$ and divides $U$ into four sections. Consider the open set $E=\{(x,y)\in\Omega|\widetilde{v}_t>0\}$. It follows from the maximum principle that each component of $E$ has to meet the boundary $\pt\Omega$. Accordingly, we see that the open set $G=\{(x,y)\in\Omega|\widetilde{v}_t<0\}$ has more than two components. This shows that $v_t$ has more than two minimal points.

``$\text{\bf{only if part}}$". Suppose that $v_t$ has more than two minimal points and there is no point $p$ with $Dv_t(p)=0$ and $K_t(p)>0$. By Lemma \ref{lem31} and Lemma \ref{lem32}, we see that each critical point of $v_t$ is a minimal point. Since $Dv_t$ does not vanish on $\pt\Omega$, then Lemma \ref{lem31} and Lemma \ref{lem32} imply that every critical point of $v_t$ is isolated and the number of critical(minimal) points is finite, say $\{P_1,\cdots,p_N\}$. Hence, we have
\begin{equation}\label{eq311}
Dv_t(x,y)\neq0, \ \ \text{for all} \ (x,y)\in\Omega-\{P_1,\cdots,p_N\}.
\end{equation}
Put $m_0=\max\{v_t(P_j)|1\leq j\leq N\}$. Consider the level set $L_m=\{(x,y)\in\Omega|v_t(x,y)<m\}$ for $m_0<m<0$. It follows from $\eqref{eq311}$ and Theorem 3.1 in \cite{milnor1963morse} that the boundary $\pt L_m$ is a smooth manifold for $m_0<m<0$ and $\{\pt L_m\}$ are diffeomorphic to each other. Since $K_t(P_j)$ is negative, if $m$ is near $m_0$, $L_m$ has more than two components. On the other hand, if $m$ is near $0$, $\pt L_m$ is diffeomorphic to $\pt\Omega$ and $L_m$ is connected. This is a contradition. So we complete the proof.
\end{proof}

Now, Lemma \ref{lem31}, Lemma \ref{lem32} and Theorem \ref{thm31} tell us the non-existence of the critical point described in the first question of the first section is equivalent to the uniqueness for the critical point of the solution to $\eqref{eq13}$, which will be proved in the next section.

\section{Proof of Theorem \ref{thm11}}
In view of Lemma \ref{lem31} and Lemma \ref{lem32} and Theorem \ref{thm31}, it suffices to show that the set of minimal point of the solution consists of only one point. Put $I=[0,1]$. Divide $I$ into two sets $I_1$ and $I_2$ as follows.
\begin{equation}\label{eq41}
\begin{split}
&I_1=\{t\in I|v_t \ \text{has only one minimal point in} \ \Omega\}\\
&I_2=\{t\in I|v_t \ \text{has more than two minimal points in} \ \Omega\}.
\end{split}
\end{equation}
Then $I=I_1+I_2$ and $I_1\cap I_2=\varnothing$. Lemma \ref{lem32} and Theorem \ref{thm31} imply that $0\in I_1$ and $I_1$ is not empty.

On the one hand, $I_2$ is open in $I$. That is, for any $t_0\in I_2$, there exists a constant $\varepsilon>0$ such that $(t_0-\varepsilon,t_0+\varepsilon)\subset I_2$. In fact, if it is not so, we can assume that there exists a sequence of solutions $\{v_{t_n}\}$ with only one minimal point and $t_n\in(t_0-\frac1n,t_0+\frac1n)$ for some positive $t_0\in I_2$. By Lemma \ref{lem31} and Theorem \ref{thm31}, $v_{t_n}$ has only one critical point. By compactness and Lemma \ref{lem31}, we take subsequence of ${v_{t_n}}$ such that
\begin{equation}
p_n\rightarrow p, \  \  Dv_{t_n}(p_n)=0, \   \  K_{t_n}(p_n)<0 ,\  \  Dv_{t_0}(p)=0, \   \  K_{t_0}(p)<0.
\end{equation}
Since $t_0\in I_2$, there exists another point $q\in U(q)\subseteq\Omega$ such that
\begin{equation}
q_n\rightarrow q, \  \ Dv_{t_n}(q_n)\rightarrow Dv_{t_0}(q)=0.
\end{equation}
By uniqueness for the critical point of $v_{t_n}$, we can take subsequence of $\{v_{t_n}\}$ such that $v_{t_n}$ are all monotone in the line $l(p_n,q_n)$. Then there exists a sequence of points $\{s_n|s_n\in l(p_n,q_n)\}$ such that
\begin{equation}
|Dv_{t_n}(s_n)|\leq |Dv_{t_n}(q_n)|\rightarrow 0, \  \  |K_{t_n}(s_n)|=\frac{|Dv_{t_n}(q_n)|}{|p_{n}-q_{n}|}\rightarrow 0.
\end{equation}
Therefore, there should be a point $s\in l(p,q)$ satisfies
\begin{equation}
Dv_{t_0}(s)=0, \  \  K_{t_0}(s)=0.
\end{equation}
This is a contradiction with the Lemma \ref{lem31}.

On the other hand, $I_2$ is closed in $I$. In fact, let $\{t_j\}$ be a sequence of points in $I_2$ such that $t_j$ converges to $t_0$ as $j$ goes to $\infty$. Theorem \ref{thm31} and the compactness imply that there exists a subsequence $\{t_k\}$, a sequence $\{p_k\}$ and a point $p\in\Omega$ such that
\begin{equation}\label{eq42}
p_k\rightarrow p \ \text{as} \ k\rightarrow\infty,\ \ \ Dv_{t_k}(p_k)=0, \ \text{and} \ K_{t_k}(p_k)>0.
\end{equation}
By continuty, we have
\begin{equation}\label{eq43}
Dv_{t_0}(p)=0, \ \text{and} \ K_{t_0}(p)\geq0.
\end{equation}
Since $Dv_{t_0}\neq0$ on $\pt\Omega$, $p\in\Omega$. Therefore it follows from Lemma \ref{lem31} and Lemma \ref{lem32}, Theorem \ref{thm31} and $\eqref{eq43}$ that $t_{0}\in I_2$. This shows that $I_2$ is closed in $I$.

Hence, $I_2$ must be $\varnothing$ or $I$. Since $I_1$ is not $\varnothing$, $I_1=I$. This completes the proof.

\section{Sharp $C^0$ and $C^1$ estimates}
In \cite{payne1979some}, the authors derived a maximum principle for a function $\Phi(x;\alpha)$ defined as
\begin{equation}\label{eq61}
\Phi(x;\alpha)=\int_0^{q^2}\frac{g(\xi)+2\xi g'(\xi)}{\rho(\xi)}d\xi+\alpha\int_0^uf(\eta)d\eta,
\end{equation}
where $g>0,\rho>0,f$ are functions and $u$ satisfies the following elliptic equation
\begin{equation}\label{eq62}
(g(q^2)u_i)_i+\rho(q^2)f(u)=0, \ \ q^2=u_iu_i=|Du|^2.
\end{equation}

In our case, we can take $g(\xi)=(1-\xi)^{-\frac12},\rho=1,f=-2H$. Then
\begin{equation}
\Phi(x;\alpha)=2(\frac{1}{\sqrt{1-|Du|^2}}-1-\alpha Hu).
\end{equation}
In particular, $\Phi:=\Phi(x;1)=2(\frac{1}{\sqrt{1-|Du|^2}}-1-Hu)$.

Theorem 4 in \cite{payne1979some} gave us
\begin{equation}
(\delta_{ij}+\frac{u_iu_j}{1-|Du|^2})\Phi_{ij}+W_k\Phi_k\geq 0,
\end{equation}
where $W_k$ is a vector function uniformly bounded in $\Omega$. It follows that $\Phi(x,1)$ takes its maximum value on $\pt\Omega$. Together with $\eqref{eq61}$, we know $\Phi(x;1)$ takes its maximum value where $|Du|^2=\max\limits_{\pt\Omega}|Du|^2$. It follows that, at any point $x\in\Omega$, we have
\begin{equation}
-Hu\leq\frac{1}{\sqrt{1-\max\limits_{\pt\Omega}|Du|^2}}-\frac{1}{\sqrt{1-|Du|^2}}.
\end{equation}
So, at the critical point, we get
\begin{equation}\label{eq65}
-Hu_{min}\leq\frac{1}{\sqrt{1-q_{max}^2}}-1,
\end{equation}
where $u_{min}=\min\limits_{\Omega}u$ and $q_{max}=\max\limits_{\pt\Omega}|Du|$.

Now, we want to derive the upper bound for $|Du|_{max}^2$. Suppose $\Phi(x;\alpha)$ attains its maximum at $p\in\pt\Omega$, then $|Du|(p)=q_{max}$. On the one hand, by strong maximum principle, we have at $p$,
\begin{equation}\label{eq66}
\frac{\pt\Phi(x;\alpha)}{\pt n}=2\frac{g+2q^2g'}{\rho}u_nu_{nn}+fu_n\geq 0,
\end{equation}
where $\frac{\pt}{\pt n}$ or a subscript $n$ denotes the outward directed normal derivative on $\pt\Omega$ and the equality holds if and only if $\Phi(x;\alpha)=\text{constant}$. On the other hand, making use of $\eqref{eq62}$ evaluated on $\pt\Omega$, we have
\begin{equation}
(g+2q^2g')u_{nn}+gKu_n+\rho f=0.
\end{equation}
Together with $\eqref{eq66}$, this leads to
\begin{equation}
\frac{\pt\Phi(x;\alpha)}{\pt n}=-(2Kgu_n^2+fu_n)\geq 0.
\end{equation}
Applying to our case, we get
\begin{equation}
 \frac{q_{max}}{\sqrt{1-q_{max}^2}}\leq\frac{H}{K(p)}\leq\frac{H}{K_{min}}.
 \end{equation}
 So
 \begin{equation}\label{eq610}
 q_{max}^2\leq\frac{H^2}{H^2+K_{min}^2}.
 \end{equation}
Therefore, the left inequality in $\eqref{eq16}$ follows from $\eqref{eq65}$ and $\eqref{eq610}$. And the equality holds if and only if the the boundary is a circle. In fact, if the equality holds, then $\Phi(x;1)=\text{constant}$ on $\pt\Omega$ from the strong maximum principle. From $\eqref{eq61}$, $u_ n=\text{constant}$ on $\pt\Omega$. So $\pt\Omega$ is a circle according to the Theorem 2 and Remark 1 in \cite{serrin1971symmetry}. Conversely, if $\pt\Omega$ is a circle, then the solution $u$ is radially symmetric. So $u_ n=\text{constant}$ on $\pt\Omega$, and then the equality in $\eqref{eq610}$ follows from the divergence theorem.

To derive the upper bound of $u_{min}$ in the same way above, we need a minimum principle for $\Phi(x;1)$. First, we need the following lemma.
\begin{lem}\label{lem61}\cite{payne1979some}
\begin{equation}
(\delta_{ij}+\frac{u_iu_j}{1-|Du|^2})\Phi_{ij}(x, \alpha)+\widehat{W_k}\Phi_k(x, \alpha)=4H^2(\alpha-1)(\alpha-2)\frac{1}{\sqrt{1-|Du|^2}},
\end{equation}
where $\widehat{W_k}$ is a vector function which is singular at the critical point of $u$.
\end{lem}
From  Lemma \ref{lem61} and Hopf maximum principle, we conclude that $\Phi(x, \alpha)$ takes its minimum value either on the boundary $\pt\Omega$, or at the unique critical point of $u$ in $\Omega$ when $\alpha\in[1,2]$. What if the second alternative happens? We answer this in the following theorem whose Euclidean version was proved by Xi-Nan Ma in \cite{ma2000sharp}.
\begin{thm}\label{thm61}
Let $u\in C^{\infty}(\overline{\Omega})$ is a solution to $\eqref{eq13}$. If $\Phi(x;1)$ attains its minimum at the unique critical point in $\Omega$, then $\Phi(x;1)$ is a constant on $\overline{\Omega}$.
\end{thm}

By Theorem \ref{thm61}, we assume $\Phi(x;1)$ takes its minimum at $p'\in\pt\Omega$, then $|Du|(p')=q_{min}=\min\limits_{\pt\Omega}|Du|$ and
\begin{equation}\label{eq611}
-Hu_{min}\geq\frac{1}{\sqrt{1-q_{min}^2}}-1,
\end{equation}
\begin{equation}
\frac{\pt\Phi}{\pt n}(p';1)\leq 0,
\end{equation}
where the equality holds if and only if $\Phi(x;1)=\text{constant}$. As before, one can also get
\begin{equation}
 \frac{q_{min}}{\sqrt{1-q_{min}^2}}\geq\frac{H}{K(p')}\geq\frac{H}{K_{max}}.
 \end{equation}
 So
 \begin{equation}\label{eq614}
 q_{min}^2\geq\frac{H^2}{H^2+K_{max}^2},
 \end{equation}
where the equality holds if and only if the the boundary is a circle. Therefore, the right inequality in $\eqref{eq16}$ follows from $\eqref{eq611}$ and $\eqref{eq614}$.

For completeness, we will prove Theorem $\eqref{thm61}$ to end this paper. Our proof is similar to that  in \cite{ma2000sharp} except for the different signs in somewhere.

\begin{proof}[Proof of Theorem \ref{thm61}]
The proof consists of four steps. Assume the unique critical point be $P\in\Omega$.

Step 1: Derivatives of $\Phi$ up to the second order vanish at $P$.
 From the proof of Theorem \ref{thm11}, we can choose the coordinate at $P$ such that
 \begin{equation}\label{eq617}
 u_1(P)=u_2(P)=0,\ \ u_{11}>0, u_{22}>0, u_{12}=0.
 \end{equation}
 By direct computation, we have
 \begin{equation}\label{eq618}
 \Phi_1=2v^{-\frac32}u_iu_{i1}-2Hu_1=0,
 \end{equation}
  \begin{equation}\label{eq619}
 \Phi_2=2v^{-\frac32}u_iu_{i2}-2Hu_2=0,
 \end{equation}
\begin{equation}
 \Phi_{11}=\frac32v^{-\frac52}(2u_iu_{i1})(2u_ju_{j1})+2v^{-\frac32}u_{i1}^2+2v^{-\frac32}u_iu_{i11}-2Hu_{11}=2u_{11}^2-2Hu_{11},
 \end{equation}
 \begin{equation}\label{eq621}
 \Phi_{12}=\frac32v^{-\frac52}(2u_iu_{i1})(2u_ju_{j2})+2v^{-\frac32}u_{i1}u_{i2}+2v^{-\frac32}u_iu_{i12}-2Hu_{12}=0,
 \end{equation}
 \begin{equation}
 \Phi_{22}=\frac32v^{-\frac52}(2u_iu_{i2})(2u_ju_{j2})+2v^{-\frac32}u_{i2}^2+2v^{-\frac32}u_iu_{i22}-2Hu_{22}=2u_{22}^2-2Hu_{22},
 \end{equation}
 where $v=1-|Du|^2$. Since $\Phi$ attains its minimum at $P$, we get
 \begin{equation}
 \Phi_{11}(P)\Phi_{22}(P)-\Phi_{12}(P)\geq 0.
 \end{equation}
 Together with $\eqref{eq617}$, we know
 \begin{equation}\label{eq624}
 u_{11}(P)=u_{22}(P)=H,
 \end{equation}
 and
 \begin{equation}\label{eq625}
 \Phi_{11}(P)=\Phi_{22}(P)=0.
 \end{equation}

Step 2: Derivatives of $\Phi$ up to the fifth order vanish at $P$.
First we claim
\begin{equation}
\Phi_{x_1^kx_2^{3-k}}(P)=0, \ \ k=0,1,2,3.
\end{equation}
By $\eqref{eq617}$, $\eqref{eq624}$ and direct calculations, we have
\begin{equation}
\begin{split}
&\Phi_{x_1^3}(P)=4Hu_{x_1^3}, \\
&\Phi_{x_1^2x_2}(P)=4Hu_{x_1^2x_2}, \\
&\Phi_{x_1x_2^2}(P)=4Hu_{x_1x_2^2}, \\
&\Phi_{x_2^3}(P)=4Hu_{x_2^3}.
\end{split}
\end{equation}
Now, by differentiating $\eqref{eq13}$, we obtain
\begin{equation}
\begin{split}
&u_{x_1^3}=-u_{x_1^2x_2},\\
&u_{x_1x_2^2}=-u_{x_2^3}.
\end{split}
\end{equation}
Together with \eqref{eq618}, \eqref{eq619}, \eqref{eq621} and \eqref{eq625}, we can expand $\Phi$ in a neighborhood of $P$:
\begin{equation}\label{eq629}
\Phi(x_1,x_2;1)-\Phi(P;1)=\frac{r^3}{3!}(\Phi_{x_1^3}(P)\cos(3\phi)+\Phi_{x_1^2x_2}(P)\sin(3\phi))+O(r^4),
\end{equation}
where $(r,\phi)$ is polar coordinate. Suppose
\begin{equation}
\sqrt{(\Phi_{x_1^3}(P))^2+(\Phi_{x_1^2x_2}(P))^2}\neq 0,
\end{equation}
then \eqref{eq629} becomes
\begin{equation}\label{eq631}
\Phi(x_1,x_2;1)-\Phi(P;1)=A_3(P)\cos[3\phi-\beta_3]r^3+O(r^4),
\end{equation}
with
\begin{equation}
\begin{split}
&A_3(P)=\frac{\sqrt{(\Phi_{x_1^3}(P))^2+(\Phi_{x_1^2x_2}(P))^2}}{3!},\\
&\cos\beta_3=\frac{\Phi_{x_1^3}(P)}{\sqrt{(\Phi_{x_1^3}(P))^2+(\Phi_{x_1^2x_2}(P))^2}},\\
&\sin\beta_3=\frac{\Phi_{x_1^2x_2}(P)}{\sqrt{(\Phi_{x_1^3}(P))^2+(\Phi_{x_1^2x_2}(P))^2}}.
\end{split}
\end{equation}
From \eqref{eq631} we conclude that $\Phi$ has at least three nodal lines forming equal angles at $P$, but Lemma \ref{lem61} tells us that $\Phi$ takes its minimum value only on $\pt\Omega$ or at $P$, which is a contradiction. Thus $A_3(P)=0$. That is,
\begin{equation}
\Phi_{x_1^kx_2^{3-k}}(P)=0, \ \ k=0,1,2,3,
\end{equation}
and
\begin{equation}
u_{x_1^kx_2^{3-k}}(P)=0, \ \ k=0,1,2,3.
\end{equation}
Using the similar argument we can show
\begin{equation}
\begin{split}
0&=\Phi_{x_1^4}(P)=6H(u_{x_1^4}(P)+3H^3) \\
&=-\Phi_{x_1^2x_2^2}(P)=6H(u_{x_1^2x_2^2}(P)+H^3)\\
&=\Phi_{x_2^4}(P)=6H(u_{x_2^4}(P)+3H^3),
\end{split}
\end{equation}
\begin{equation}
\begin{split}
0&=\Phi_{x_1^3x_2}(P)=6Hu_{x_1^3x_2}(P)\\
&=-\Phi_{x_1x_2^3}(P)=6Hu_{x_1x_2^3}(P),\\
\end{split}
\end{equation}
\begin{equation}
\begin{split}
&u_{x_1^4}(P)=u_{x_2^4}(P)=-3H^3, \\
&u_{x_1^3x_2}(P)=u_{x_1x_2^3}(P)=0,\\
&u_{x_1^2x_2^2}(P)=-H^3,
\end{split}
\end{equation}
\begin{equation}
\Phi_{x_1^kx_2^{5-k}}(P)=u_{x_1^kx_2^{5-k}}(P)=0, \ \ k=0,1,2,3,
\end{equation}
and
\begin{equation}
\begin{split}
&\Phi_{x_1^5}(P)=-\Phi_{x_1^3x_2^2}(P)=\Phi_{x_1x_2^4}(P),\\
&\Phi_{x_1^4x_2}(P)=-\Phi_{x_1^2x_2^3}(P)=\Phi_{x_2^5}(P).
\end{split}
\end{equation}

Step 3: Now we assume all derivatives of $\Phi$ up to the n-th order vanish at $P$, where $n\geq 5$. Using the same argument as in the previous step, we have following relations.

If $n=2l$, $l\geq3$. Then
\begin{equation}\label{eq640}
\begin{split}
&u_{x_1^mx_2^{k-m}}(P)=u_{x_1^{k-m}x_2^m}(P)\\
&\text{for any} \ \ m=0,1,2,\ldots,k, \ \text{if} \ \ k=5,6,8,\ldots,2l,
\end{split}
\end{equation}
\begin{equation}
\begin{split}
&u_{x_1^mx_2^{k-m}}(P)=0\\
&\text{for any} \ \ m=0,1,2,\ldots,k, \ \text{if} \ \ k=5,7,9,\ldots,2l-1,
\end{split}
\end{equation}
\begin{equation}
\begin{split}
&u_{x_1^mx_2^{2p-m}}(P)=0\\
&\text{for any} \ \ m=1,3,5,\ldots,2p-1, \ \text{if} \ \ p=3,4,5,\ldots,l,
\end{split}
\end{equation}
\begin{equation}
\begin{split}
&u_{x_1^{2p}}(P)=(-1)^{p+1}(2p-1)[(2p-3)(2p-5)\ldots1]^2H^{2p-1}\\
&\text{for any} \ \ p=3,4,5,\ldots,l.
\end{split}
\end{equation}
When $l$ is even, we have for any $p=4,6,8,\ldots,l$
\begin{equation}\label{eq644}
\begin{split}
u_{x_1^{2p}}(P)\div u_{x_1^{2p-2}x_2^2}(P)&=(2p-1)\div1\\
u_{x_1^{2p-2}x_2^2}(P)\div u_{x_1^{2p-4}x_2^4}(P)&=(2p-3)\div3\\
&\vdots\\
u_{x_1^{p+2}x_2^{p-2}}(P)\div u_{x_1^px_2^p}(P)&=(p+1)\div(p-1),
\end{split}
\end{equation}
and for any $p=3,5,7,\ldots,l-1$, we have
\begin{equation}\label{eq645}
\begin{split}
u_{x_1^{2p}}(P)\div u_{x_1^{2p-2}x_2^2}(P)&=(2p-1)\div1\\
u_{x_1^{2p-2}x_2^2}(P)\div u_{x_1^{2p-4}x_2^4}(P)&=(2p-3)\div3\\
&\vdots\\
u_{x_1^{p+3}x_2^{p-3}}(P)\div u_{x_1^{p+1}x_2^{p-1}}(P)&=(p+2)\div(p-2).
\end{split}
\end{equation}
When $l$ is odd, we have the similar relation \eqref{eq644} and \eqref{eq645}.

If $n=2l+1$, $l\geq2$, by a similar argument we have \eqref{eq640}-\eqref{eq645} and
\begin{equation}
u_{x_1^mx_2^{2l+1-m}}(P)=0, \ \ \text{for any} \ \ m=0,1,2,\ldots,2l+1.
\end{equation}

Step 4: Derivatives of $\Phi$ of order n+1 vanish at $P$. We divide it into two parts according to whether $n$ is odd or even.

Case A: If $n=2l$, $l\geq 3$. By the inductive assumption, we have 
\begin{equation}
\begin{split}
&v_{x_1^mx_2^{k-m}}(P)=0\\
&\text{for any} \ \ m=0,1,2,\ldots,k, \ \ \text{if} \ \ k=1,3,5,\ldots,n-1.
\end{split}
\end{equation}
Then for any $m=0,1,2,\ldots,n+1$
\begin{equation}
\begin{split}
&(2v^{-\frac12})_{x_1^mx_2^{n+1-m}}(P)=-v^{\frac32}v_{x_1^mx_2^{n+1-m}}(P)\\
&=2v^{-\frac32}((n+1-m)Hu_{x_1^mx_2^{n+1-m}}+mHu_{x_1^mx_2^{n+1-m}})\\
&=2(n+1)Hu_{x_1^mx_2^{n+1-m}}.
\end{split}
\end{equation}
So 
\begin{equation}
\Phi_{x_1^mx_2^{n+1-m}}(P)=2nHu_{x_1^mx_2^{n+1-m}}(P).
\end{equation}
Now, by differentiating \eqref{eq13}, we obtain
\begin{equation}
u_{x_1^mx_2^{n+1-m}}(P)=-u_{x_1^{m+2}x_2^{n-1-m}}(P),\ \ \text{for} \ \ m=0,1,2,\ldots,n+1.
\end{equation}
Then 
\begin{equation}
\Phi_{x_1^mx_2^{n+1-m}}(P)=-\Phi_{x_1^{m+2}x_2^{n-1-m}}(P),\ \ \text{for} \ \ m=0,1,2,\ldots,n+1.
\end{equation}
Using Taylor expansion as in Step 2, we can conclude that the derivatives of $\Phi$ of order $n+1$ vanish at $P$.

Case B: If $n=2l+1$, $l\geq2$, so $n+1=2(l+1)$ is even. We first look for the relations among $\Phi_{x_1^mx_2^{n+1-m}}(P)$, where $m=0,2,4,\ldots,n+1$. Through computations, we have
\begin{equation}
\Phi_{x_1^{n+1}}(P)=2nH(u_{x_1^{n+1}}+(-1)^{l+1}(2l+1)[(2l-1)(2l-3)\ldots1]^2H^{2l+1}),
\end{equation}
and
\begin{equation}
\Phi_{x_1^{n-1}x_2^2}(P)=2nH(u_{x_1^{n-1}x_2^2}+(-1)^{l+1}[(2l-1)(2l-3)\ldots1]^2H^{2l+1}).
\end{equation}
Now, by differentiating \eqref{eq13}, we get
\begin{equation}
(\Delta u+u_iu_ju_{ij}v^{-1})_{x_1^{n-1}}(P)=(2Hv_{\frac12})_{x_1^{n-1}}(P).
\end{equation}
Together with the relations in Step 3, this leads to 
\begin{equation}
u_{x_1^{n+1}}+u_{x_1^{n-1}x_2^2}=(n+1)(-1)^l[(2l-1)(2l-3)\ldots1]^2H^{2l+1}.
\end{equation}
So 
\begin{equation}
\Phi_{x_1^{n+1}}(P)=-\Phi_{x_1^{n-1}x_2^2}(P).
\end{equation}
By a similar argument, it follows that
\begin{equation}
\Phi_{x_1^mx_2^{n+1-m}}(P)=-\Phi_{x_1^{m+2}x_2^{n-1-m}}(P),\ \ \text{for} \ \ m=0,2,4,\ldots,n+1.
\end{equation}
Then, using the same argument, we have
\begin{equation}
\Phi_{x_1^mx_2^{n+1-m}}(P)=-\Phi_{x_1^{m+2}x_2^{n-1-m}}(P),\ \ \text{for} \ \ m=0,1,2,\ldots,n+1.
\end{equation}
Now, as in Case A, we can show the derivatives of $\Phi$ of order $n+1$ vanish at $P$.

According to the unique continuation of analytic function, we know if $\Phi$ attains its minimum at $P$, then it must be a constant. This completes the proof.

\section{Acknowledgements}
The author is fully supported by China Scholarship Council(CSC) for visiting University of California, Santa Cruz. The author is also very grateful to his advisors Professor Xi-Nan Ma and Professor Jie Qing for their expert guidances and useful conversations.

\end{proof}
\nocite{*}

\bibliographystyle{amsplain}
\bibliography{reference}
\end{document}